\documentclass[11pt,twoside]{amsart}

\usepackage{hyperref}

\usepackage{amssymb}

\usepackage{array}

\usepackage{graphicx,color}

\usepackage{geometry}
\newtheorem{theorem}{Theorem}[section]
\newtheorem{proposition}[theorem]{Proposition}

\newtheorem{lemma}[theorem]{Lemma}
\newtheorem{corollary}[theorem]{Corollary}
\newtheorem{definition}[theorem]{Definition}
\newtheorem{notation}[theorem]{Notation}
\newtheorem{example}[theorem]{Example}
\newtheorem{remark}[theorem]{Remark}
\newtheorem{conjecture}[theorem]{Conjecture}
\newtheorem{problem}[theorem]{Problem}
\newtheorem{question}[theorem]{Question}

\newcommand{\ZZ}{\mathbb{Z}}

\newcommand{\NN}{\mathbb{N}}

\newcommand {\PP}{\mathbb{P}}

\newcommand{\kk}{\mathbb{K}}

\newcommand{\cA}{\mathcal{A}}
\newcommand{\cB}{\mathcal{B}}

\newcommand{\cL}{\mathcal{L}}
\newcommand{\cM}{\mathcal{M}}

\newcommand{\cO}{\mathcal{O}}

\DeclareMathOperator{\syz}{syz}

\DeclareMathOperator{\HH}{H}
\DeclareMathOperator{\hh}{h}

\DeclareMathOperator{\Ext}{Ext}

\DeclareMathOperator{\rk}{rk}

\DeclareMathOperator{\relint}{Relint}

\begin{document}

\title[Stability and rigidness of syzygy bundles]{Syzygy bundles of non-complete linear systems: \\
stability and rigidness}
  \author[R.\ M.\ Mir\'o-Roig]{Rosa M.\ Mir\'o-Roig} 
  \address{Facultat de
  Matem\`atiques i Inform\`atica, Universitat de Barcelona, Gran Via des les
  Corts Catalanes 585, 08007 Barcelona, Spain} \email{miro@ub.edu, ORCID 0000-0003-1375-6547}
  \author[Marti Salat-Molt\'o]{Marti Salat-Molt\'o}
\address{Department de matem\`{a}tiques i Inform\`{a}tica, Universitat de Barcelona, Gran Via de les Corts Catalanes 585, 08007 Barcelona,
Spain}
\email{marti.salat@ub.edu}



\thanks{The first author has been partially supported by the grant PID2019-104844GB-I00. The second author has been partially supported by the grant MDM-2014-0445-18-2.}

\begin{abstract} 
Let $(X,L)$ be a polarized smooth projective variety. For any basepoint-free linear system $\cL_{V}$ with $V\subset\HH^{0}(X,\cO_{X}(L))$ we consider the syzygy bundle $M_{V}$ as the kernel of the evaluation map $V\otimes \cO_{X}\rightarrow \cO_{X}(L)$. The purpose of this article is twofold. First, we assume that $M_{V}$ is $L$-stable and prove that, in a wide family of projective varieties, it represents a smooth point $[M_{V}]$ in the corresponding moduli space $\cM$. We  compute the dimension of the irreducible component of $\cM$ passing through $[M_{V}]$ and whether it is an isolated point. It turns out that the rigidness of $[M_{V}]$ is closely related to the completeness of the linear system $\cL_{V}$. In the second part of the paper, we address a question posed by Brenner regarding the stability of $M_{V}$ when $V$ is general enough. We answer this question for a large family of polarizations of $X=\PP^{m}\times\PP^{n}$.
\end{abstract}

\maketitle

\section{Introduction}
Given a smooth projective variety $X$ and a very ample line bundle $L$ on $X$, let $V\subset\HH^{0}(X,\cO_{X}(L))$ be a subspace such that the corresponding linear system $\cL_{V}$ is basepoint-free. The projective morphism
\[
\phi_{V}:X\longrightarrow\PP(V^{\ast})
\]
is a central object of study in algebraic geometry. In particular, the syzygies of $\phi_{V}$ are encoded in the so-called syzygy bundle $M_{V}$, whose study has been of increasing interest in the last decades. Namely, we define the syzygy bundle $M_{V}$ as the kernel of the evaluation map $ev:V\otimes \cO_{X}\rightarrow \cO_{X}(L)$,
which is surjective. In particular we have the following short exact sequence
\[
0\longrightarrow M_{V}\longrightarrow V\otimes \cO_{X}\longrightarrow \cO_{X}(L)\longrightarrow 0,
\]
which describes $M_{V}$ as a vector bundle of rank $\dim V-1$.
The syzygy bundles $M_{V}$ have been studied from many perspectives in the last decades. In particular when $V=\HH^{0}(X,\cO_{X}(L))$, the linear system $\cL_{V}$ is complete and the morphism $\phi_{V}$ coincides with the embedding 
\[
\phi_{L}:X\hookrightarrow \PP(\HH^{0}(X,\cO_{X}(L))^{\ast}),
\]
given by the very ample line bundle $L$. In this case, the syzygy bundle is denoted by $M_{L}:=M_{V}$, and it is behind many geometric properties of the embedding $\phi_{L}$. For instance, the properties $(N_{p})$ in the sense of Green \cite{G}, or the stability of the pullback $\phi_{L}^{\ast}T_{\PP^{N_{L}}}$ of the tangent bundle of $\PP^{N_{L}}$, which is related to the stability of $M_{L}$ and has been studied thoroughly in the recent years \cite{CL, EL, ELM, HMP, MS, MR, T} and has led to the so-called Ein-Lazarsfeld-Mustopa Conjecture (see Conjecture \ref{Conjecture:stable}). In this work, we focus our attention  on the stability of the syzygy bundles $M_{V}$ arising from  linear systems $\cL_{V}$ which are non-necessarily complete.

In \cite{Brenner}, motivated by the theory of tight closure, the systematic study of $L$-stable syzygy bundles on $\PP^{n}$ was considered. In particular, in \cite[Question 7.8]{Brenner} Brenner asked the following question:
\begin{question}\label{Question:Brenner_intro}
Let us consider integers $n,d\geq 1$. For which integers $r$ such that $n+1\leq r\leq \binom{n+d}{d}$ there exist $r$ monomials $m_{1},\dotsc,m_{r}$ with no common factors such that the syzygy bundle $M_{V}$ corresponding to the subspace $V=\langle m_{1},\dotsc,m_{r}\rangle\subset\HH^{0}(\PP^{n},\cO_{\PP^{n}}(d))$ is semistable?
\end{question}
The case $r=\binom{n+d}{n}$ had been previously proved in \cite{Fle}, and a complete answer was given in \cite{CMM} and \cite{C}. Moreover, in \cite{CMM} the authors studied the local geometry of the moduli space $\cM$ in which an $L$-stable syzygy bundle may be represented. In \cite[Theorem 4.4]{CMM} they proved that apart from few exceptions, an $L-$stable syzygy bundle $M_{V}$ correspond to a smooth point in $\cM$ and they computed the dimension of the irreducible component containing it. From their result one can see that if $n\geq4$, then $M_{V}$ is infinitesimally rigid if and only if $r=\binom{n+d}{n}$ and the linear system is complete. As we show in Theorem \ref{Theorem:moduli}, this is not a particular feature of syzygy bundles on projective spaces, and we can generalize this fact to a large family of smooth projective varieties.

\vskip 2mm 
In this paper, we consider the analogous of Brenner's Question \ref{Question:Brenner_intro} for any smooth projective variety:
\begin{question}\label{Question:Brenner_intro_projective}\rm
Let us fix an ample line bundle $L$ on a smooth projective variety $X$ of dimension $d$. 
For which integers $\dim(X)+1\leq r\leq  \dim\HH^{0}(X,L)$,  is there a basepoint-free linear system $\cL_{V}$ associated to a subspace $V\subset\HH^{0}(X,L)$ of dimension $r$, such that the syzygy bundle $M_{V}$ is $L-$stable?
\end{question}
Since stability is an open property, we notice that positively answering Question \ref{Question:Brenner_intro_projective} for a certain integer $r$ we obtain that the syzygy bundle $M_{V}$ corresponding to a {\em general} subspace $V\subset\HH^{0}(X,\cO_{X}(L))$ of dimension $r$ is $L$-stable. Notwithstanding, the implications of Question \ref{Question:Brenner_intro_projective} go beyond this fact: it sheds new light on the geometry of certain moduli spaces of $L$-stable vector bundles on a projective variety. To be more precise, in Section \ref{Section:moduli} we consider a large family of polarized smooth projective varieties $(X,L)$ of any dimension. In this setting, we show (see Theorem \ref{Theorem:moduli}) that a positive answer to Question \ref{Question:Brenner_intro_projective} automatically yields smooth points on a suitable moduli space. Even more, we show that in this case the dimension of the irreducible component containing these points is fully described by only using the very ample line bundle $L$.
It is worthwhile to mention that this family of projective varieties include smooth complete projective toric varieties, Grassmannians, flag varieties among others. On the other hand, Theorem \ref{Theorem:moduli} works under well understood assumptions on the very ample line bundle $L$, which are actually very mild hypothesis when $\dim(X)\neq 3$. 

Motivated by these facts, we devote the second half of this paper to answer Question \ref{Question:Brenner_intro_projective} for the product of two projective spaces $X=\PP^{m}\times\PP^{n}$, which is an example of a smooth complete projective toric variety. In this case, we can use the Cox ring of $X$ which is the standard-bigraded polynomial ring $\kk[x_{0},\dotsc,x_{m},y_{0},\dotsc,y_{n}]$ and we examine the possible degrees of syzygies among $r$ forms $\{f_{1},\dotsc,f_{r}\}$ of degree $(a,b)$ with $a,b>0$. This allows us to give a positive answer to Question \ref{Question:Brenner_intro_projective} in a large amount of cases (see Theorem \ref{Theorem:main}).

This work is organized as follows. In Section \ref{Section:Preliminaries} we gather the basic results regarding stability of vector bundles on polarized projective varieties $(X,L)$ and the theory of syzygy bundles of linear systems needed in the sequel. Afterwards, the paper is divided in two main sections. In Section \ref{Section:moduli} we consider a large family of polarized projective varieties $(X,L)$, which include, but is not limited to, smooth complete projective toric varieties, Grassmannians or flag varieties. We prove (Theorem \ref{Theorem:moduli}) that in this setting an $L$-stable syzygy bundle of a non-necessarily complete linear system corresponds to a smooth point in its moduli space, and we give explicitly the dimension of the irreducible component containing that point. The second part of this work is found in Section \ref{Section:stability of non-complete}, where we aim to answer Question \ref{Question:Brenner_intro_projective} for products of projective spaces. Our main results in this regard are Theorem \ref{Theorem:main} and Corollary \ref{Corollary:PnxPm} which answers Question \ref{Question:Brenner_intro_projective} in a large number of cases. In Theorem \ref{Theorem:moduli_PnxPm} we apply this results to give insight on the moduli spaces of syzygy bundles on $\PP^{m}\times\PP^{n}$. Finally, we end this section posing some open questions regarding the stability of syzygy bundles of non-complete linear systems.

\vskip 5mm \noindent
 {\bf Acknowledgements.} The authors are thankful to the anonymous referee for useful comments. The second author is grateful to Liena Colarte G\'omez for useful discussions regarding syzygies of bigraded homogeneous forms.

\section{Basic results}\label{Section:Preliminaries}

Let $(X,L)$ be a polarized smooth projective variety of dimension $d$, defined over an algebraically closed field $\kk$ of characteristic zero and let 
 $L$ be a  globally generated line bundle. For any vector subspace $V\subset \HH^{0}(X,L)$ we denote by $\cL_{V}$ the corresponding linear system. If $\cL_{V}$ is base point-free, we define the {\em syzygy bundle} $M_{V}$ as the kernel of the evaluation map 
 \[
 ev:V\otimes {\cO}_X \longrightarrow \cO_X (L).
 \]
  Notice that $M_{V}$ is a vector bundle fitting in the following short exact sequence 
\begin{equation}\label{Eq:Syzygy bundle exact sequence}
0 \longrightarrow M_V \longrightarrow V\otimes {\cO}_X \longrightarrow \cO_X(L)\longrightarrow 0. \end{equation}
In particular we have:

\vskip 2mm
\begin{itemize}
    \item $c_1(M_V)=-c_1(L)$,
    \item $\rk(M_V)=\dim(V)-1$,
    \item $\mu_{L}(M_V)=\frac{-L^d}{\dim(V)-1}.$
\end{itemize}

\vskip 4mm
The goal of this paper is to study the stability of the vector bundle $M_{V}$ for appropriate subspaces $V\subset\HH^{0}(X,\cO_{X}(L))$ and to obtain local information on the geometry of their corresponding moduli spaces. Let us first recall the definition and some key result about the stability of vector bundles.

\begin{definition}\rm Let $(X,L)$ be a polarized smooth variety of dimension $d$. A vector bundle $E$ on $X$ is $L-$stable (resp. $L-$semistable) if for any subsheaf $F\subset E$ with $0<\rk(F)<\rk(E)$, we have 
\[
\mu_{L}(F):=\frac{c_1(F)L^{d-1}}{\rk(F)}<\mu_{L}(E):=\frac{c_1(E)L^{d-1}}{\rk(E)} \]
\[\text{(resp.}\;\mu_{L}(F):=\frac{c_1(F)L^{d-1}}{\rk(F)}\le \mu_{L}(E):=\frac{c_1(E)L^{d-1}}{\rk(E)}\;).
\]
\end{definition}

\vskip 2mm
The following result is a cohomological characterization of the stability, and it will play a central role in the proof of our main result.

\begin{lemma}  \label{key} \cite[Lemma 2.1]{C} Let $(X,L)$ be a polarized smooth variety of dimension $d$. Let  $E$ be a vector bundle on $X$. Suppose  that for any integer $q$ and any line bundle $G$ on $X$ such that
$$ 0<q<\rk(E) \quad \text{ and } \quad (G\cdot L^{d-1})\ge q\mu _L(E) $$
one has
$ \HH ^0(X,\bigwedge^qE\otimes G^{\vee})=0.
$
Then, $E$ is $L-$stable.
\end{lemma}

\begin{remark}\rm
A vector bundle $E$ satisfying the hypothesis of Lemma \ref{key} is said to be {\em cohomologically stable}. It is worthwhile to point out that any cohomological stable vector bundle on a polarized variety  $(X,L)$ is $L$-stable but not vice versa.
\end{remark}

The stability of syzygy bundles associated to complete linear systems (i.e. when $V=\HH^{0}(X,\cO_{X}(L))$) on polarized varieties $(X,L)$ has received a lot of attention on the last decades (see, for instance, \cite{CL, EL, ELM, Fle, MR, T}). Our goal is to answer the following much more general question:
\begin{question}\label{Question:Brenner}\rm
Let us fix an ample line bundle $L$ on a smooth projective variety $X$ of dimension $d$. 
For which integers $r\leq  \dim\HH^{0}(X,L)$,  is there a base point-free linear system $\cL_{V}$ associated to a subspace $V\subset\HH^{0}(X,L)$ of dimension $r$, such that the syzygy bundle $M_{V}$ is $L-$stable?
\end{question}

Question \ref{Question:Brenner} is a generalization of a question raised by Brenner in \cite[Question 7.8]{Brenner}, regarding the stability of syzygy bundles of non-complete linear systems in $\PP^{N}$. This problem has been further studied in \cite{CMM, MM, C}, where a complete answer for the case $(X,L)=(\PP^{N},\cO_{\PP^{N}}(d))$ is given. 

\begin{remark}\label{Remark:stability open property}\rm
i) Since the $L-$stability is an open property, Question \ref{Question:Brenner} is equivalent to ask for which integers $r\leq \HH^{0}(X,L)$, the syzygy bundle  $M_{V}$ corresponding to a {\em general} base point-free linear system $\cL_{V}$ given by a subspace $V\subset \HH^{0}(X,L)$ of dimension $r$, is $L-$stable.

ii) In the case $V=\HH^{0}(X,L)$ there is a conjecture by Ein, Lazasferld and Mustopa (see \cite[Conjecture 2.6]{ELM} or Conjecture \ref{Conjecture:stable}) which addresses the stability of the syzygy bundle $M_{V}$.
\end{remark}
\begin{remark}\rm
As shown in Section \ref{Section:moduli}, answering Question \ref{Question:Brenner}, and thus providing general syzygy bundles which are {\em $L-$stable}, shed new light on the geometry of the moduli spaces where the syzygy bundles are represented as a point.
\end{remark}

\section{Rigidness of the syzygy bundles}\label{Section:moduli}
In this section we focus on a polarized projective variety $(X,L)$ of dimension $d$ and we consider a syzygy bundle $M_{V}$ associated to a  base point-free linear system $\cL_{V}$ with $V\subset \HH^{0}(X,L)$ a subspace of dimension $r\leq \hh^{0}(X,\cO_{X}(L))$. In general it is not known if $M_{V}$ is $L-$stable (see Conjecture \ref{Conjecture:stable} and Question \ref{question}). Notwithstanding, if $M_{V}$ is $L$-stable, then it represents a point inside a suitable moduli space $\cM=\cM_{X}(r-1;c_{1},\dotsc,c_{\min\{r-1,d\}})$ of rank $r-1$ $L-$stable vector bundles with Chern classes $c_{i}=c_{i}(M_{V})$ for $1\leq i\leq \min\{r-1,d\}$. In this section we assume that $M_{V}$ is $L$-stable, and we study the geometry of this moduli space $\cM$ around $[M_{V}]$.

Recall that the Zariski tangent space of $\cM$ at a point $[E]$ is canonically given by
\[
T_{[E]}\cM\cong \Ext^{1}(E,E)\cong \HH^{1}(X,E\otimes E^{\vee}),
\]
and we say that $E$ is infinitesimally rigid if $[E]$ is an isolated point, or equivalently if $\dim T_{[M_{V}]}\cM=0$. We have the following result:

\begin{theorem}\label{Theorem:moduli}
Let $(X,L)$ be a polarized projective variety such that $\HH^{1}(X,\cO_{X})=\HH^{2}(X,\cO_{X})=\HH^{3}(X,\cO_{X})=0$ and $\HH^{1}(X,\cO_{X}(L))=0$. For any base point-free linear system $\cL_{V}$ with $V\subset \HH^{0}(X,L)$, denote by $M_{V}$ the corresponding syzygy bundle. If $M_{V}$ is $L-$stable and either
\begin{itemize}
    \item[a)] $\dim(X)\geq 4$,
    \item[b)] $\dim(X)=3$ and the linear system $\cL_{V}$ is complete ($V=\HH^{0}(X,\cO_{X}(L))$).
    \item[c)] $\dim(X)=3$, the linear system $\cL_{V}$ is non-complete and $\HH^{3}(X,\cO_{X}(-L))=0$, or
    \item[d)] $\dim(X)=2$,
\end{itemize}
then
\begin{itemize}
\item[i)] $[M_{V}]\in\cM$ is a smooth point.
\item[ii)] The dimension of the irreducible component in $\cM$ containing $[M_{V}]$ is either
\begin{itemize}
\item[] $\dim_{\kk}T_{[M_{V}]}\cM = r(\hh^{0}(X,\cO_{X}(L))-r)$, in Cases a), b), c); or
\item[] $\dim_{\kk}T_{[M_{V}]}\cM = r(\hh^{0}(X,\cO_{X}(L))-r)+r\hh^{2}(X,\cO_{X}(-L))$, in Case d).
\end{itemize}

\end{itemize}
In particular, when $\dim(X)\geq3$, $M_{V}$ is infinitesimally rigid if and only if $V=\HH^{0}(X,\cO_{X}(L))$.
\end{theorem}
\begin{proof}
Let us start studying $\HH^{2}(X,M_{V}\otimes M_{V}^{\vee})$. We consider the exact sequence
\begin{equation}\label{Eq:exact sequence M_L}
0\longrightarrow M_{V}\longrightarrow \cO_{X}^{r}\longrightarrow \cO_{X}(L)\longrightarrow 0.
\end{equation}
Dualizing the exact sequence (\ref{Eq:exact sequence M_L}) and tensoring it by $M_{V}$, we obtain:
\begin{equation}\label{Eq:exact sequence M_L**M_L^v}
0\longrightarrow M_{V}(-L)\longrightarrow M_{V}\otimes \cO^{r}_{X}\longrightarrow M_{V}\otimes M_{V}^{\vee}\longrightarrow 0.
\end{equation}
From the exact sequence of cohomology of (\ref{Eq:exact sequence M_L**M_L^v}) we have
\begin{equation}\label{Eq:H^2(M_L**M_L^v)}
\dotsb\longrightarrow \HH^{2}(X,M_{V}^{r})\longrightarrow \HH^{2}(X,M_{V}\otimes M_{V}^{\vee})\longrightarrow \HH^{3}(X,M_{V}(-L))\longrightarrow\dotsb
\end{equation}
To see that $\HH^{2}(X,M_{V}\otimes M_{V}^{\vee})=0$ it is enough to see that $\HH^{2}(X,M_{V})=\HH^{3}(X,M_{V}(-L))=0$.
From the exact sequence of cohomology of (\ref{Eq:exact sequence M_L}) and the hypothesis  $\HH^{1}(X,\cO_{X}(L))=\HH^{2}(X,\cO_{X})=0$ we get
\begin{equation}\label{Eq:H^2(M_L)}
\HH^{2}(X,M_{V})=0.
\end{equation}

Let us compute $\HH^{3}(X,M_{V}(-L))$.  From the exact sequence (\ref{Eq:exact sequence M_L}) tensored by $\cO_{X}(-L)$ we have 
\begin{equation}\label{Eq:H^3(M_L(-L)) exact sequence}
\dotsb\longrightarrow \HH^{2}(X,\cO_{X})\longrightarrow \HH^{3}(X,M_{V}(-L))\longrightarrow \HH^{3}(X,\cO_{X}(-L)^{r})\rightarrow\HH^{3}(X,\cO_{X})\longrightarrow\dotsb.
\end{equation}
Since we assume that $\HH^{2}(X,\cO_{X})=\HH^{3}(X,\cO_{X})=0$ we have that
\begin{equation}\label{Eq:H^3(M_L(-L))}
    \HH^{3}(X,M_{V}(-L))\cong\HH^{3}(X,\cO_{X}(-L)).
\end{equation}
We leave Case b) to the end of the proof. In any other case we have $\HH^{3}(X,\cO_{X}(-L))=0$ and it follows that
\[
\HH^{2}(X,M_{V}\otimes M_{V}^{\vee})=0.
\]
Therefore $[M_{V}]$ is a smooth point in $\cM$.
Let us compute $T_{[M_{V}]}\cM\cong \HH^{1}(X,M_{V}\otimes M_{V}^{\vee})$.

Since $V\subset\HH^{0}(X,\cO_{X}(L))$, the map $(\HH^{0}(ev):\HH^{0}(X,\cO_{X}^{r})\rightarrow \HH^{0}(X,\cO_{X}(L)))$ is injective. Hence, 
\begin{equation}\label{Eq:M_V no global sections}
\HH^{0}(X,M_{V})=0,
\end{equation}
From (\ref{Eq:exact sequence M_L**M_L^v}) and using (\ref{Eq:H^2(M_L)}) we have that
\begin{multline}\label{Eq:H^1(M_L**M_L^v)}
0
\longrightarrow\HH^{0}(X,M_{V}\otimes M_{V}^{\vee})
\longrightarrow\HH^{1}(X,M_{V}(-L))
\longrightarrow\HH^{1}(X,M_{V}^{r})\\
\longrightarrow \HH^{1}(X,M_{V}\otimes M_{V}^{\vee})
\longrightarrow \HH^{2}(X,M_{V}(-L))\rightarrow0.
\end{multline}

Since $M_{V}$ is $L-$stable, in particular, it is simple. So we have
\begin{equation}\label{Eq:H^0(M_L**M_L^v)}
\HH^{0}(X,M_{V}\otimes M_{V}^{\vee})\cong\kk.
\end{equation}
On the other hand, twisting (\ref{Eq:exact sequence M_L}) we obtain
\begin{multline*}
0\longrightarrow\HH^{0}(X,\cO_{X})
\longrightarrow \HH^{1}(X,M_{V}(-L))
\longrightarrow\HH^{1}(X,\cO_{X}(-L)^{r})
\longrightarrow \HH^{1}(X,\cO_{X})\\
\longrightarrow
\HH^{2}(X,M_{V}(-L))
\longrightarrow \HH^{2}(X,\cO_{X}(-L)^{r})
\longrightarrow\HH^{2}(X,\cO_{X})
\longrightarrow\dotsb
\end{multline*}
Since $\HH^{1}(X,\cO_{X})=\HH^{2}(X,\cO_{X})=0$ and by Kodaira's vanishing theorem $\HH^{1}(X,\cO_{X}(-L))=0$, we have that
\begin{equation}\label{Eq:H^i(M_L(-L))}
\begin{array}{l}
    \HH^{1}(X,M_{V}(-L))\cong\HH^{0}(X,\cO_{X})\\
    \HH^{2}(X,M_{V}(-L))\cong\HH^{2}(X,\cO_{X}(-L)^{r}).
\end{array}
\end{equation}
Finally, from (\ref{Eq:exact sequence M_L}) we have that
\begin{equation}\label{Eq:H^1(M_L) exact sequence}
0\longrightarrow \HH^{0}(X,M_{V})
\longrightarrow \HH^{0}(X,\cO_{X}^{r})
\longrightarrow \HH^{0}(X,\cO_{X}(L))
\longrightarrow \HH^{1}(X, M_{V})
\longrightarrow 0.
\end{equation}
Since $\HH^{0}(X,M_{V})=0$ (see \ref{Eq:M_V no global sections}), we obtain that
\[
    \HH^{0}(X,\cO_{X}(L))\cong\HH^{1}(X,M_{V})\oplus \HH^{0}(X,\cO^{r}_{X}),
\]
which yields
\begin{equation}\label{Eq:H^1(M_L)}
\hh^{1}(X,M_{V})=\hh^{0}(X,\cO_{X}(L))-\hh^{0}(X,\cO_{X}^{r})=\hh^{0}(X,\cO_{X}(L))-r.
\end{equation}

Finally, from (\ref{Eq:H^1(M_L**M_L^v)}) and using (\ref{Eq:H^0(M_L**M_L^v)}), (\ref{Eq:H^i(M_L(-L))}) and (\ref{Eq:H^1(M_L)})  we get:
\begin{align*}\label{Eq:HH^1(M_L**M_L^v) comput}
\hh^{1}(X,M_{V}\otimes M_{V}^{\vee})
&=\hh^{2}(X,M_{V}(-L))+r\hh^{1}(X,M_{V})-\hh^{1}(X,M_{V}(-L))+\hh^{0}(X,M_{V}\otimes M_{V}^{\vee})\\
	&=\hh^{2}(X,\cO_{X}(-L)^{r})+r(\hh^{0}(X,\cO_{X}(L))-\hh^{0}(X,\cO_{X}^{r}))-\hh^{0}(X,\cO_{X})+1\\
	&=r\hh^{2}(X,\cO_{X}(-L))+r(\hh^{0}(X,\cO_{X}(L))-r).
\end{align*}
In particular, for Case d) ($\dim(X)=2$), the proof is finished.

On the other hand, notice that if $\dim(X)\geq 3$, by Kodaira's vanishing theorem we have $\HH^{2}(X,\cO_{X}(-L))=0$. Thus, we have
\begin{equation}\label{Eq:H^1(M_L**M_L^v dim>=3}
\hh^{1}(X,M_{V}\otimes M_{V}^{\vee})=r(\hh^{0}(X,\cO_{X}(L))-r)
\end{equation}
the proof now follows for Case a) ($\dim(X)\geq 4$ and Case c) ($\dim(X)=3$, $V\subsetneq\HH^{0}(X,\cO_{X}(L))$ and $\HH^{3}(X,\cO_{X}(-L))=0$).

To finish the proof we need to tackle Case b), that is when $\dim(X)=3$ and $V=\HH^{0}(X,\cO_{X})$. In this case we have not seen that $[M_{V}]$ is a smooth point in $\cM$ so we cannot deduce directly that $\dim_{\kk} T_{[M_{V}]}$ gives the dimension of the irreducible component of $\cM$ containing $[M_{V}]$. However, from (\ref{Eq:H^1(M_L**M_L^v dim>=3}) we obtain that
\[
\hh^{1}(X,M_{V}\otimes M_{V}^{\vee})=r(\hh^{0}(X,\cO_{X}(L))-r)=0
\]
since in this case $r=\hh^{0}(X,\cO_{X}(L))$. Therefore $\dim_{\kk} T_{[M_{V}]}=0$, so $M_{V}$ is infinitesimally rigid and so $[M_{V}]$ is smooth.
\end{proof}
\begin{remark}
\rm i) There is a wide range of projective varieties $X$ satisfying that $\HH^{1}(X,\cO_{X})=\HH^{2}(X,\cO_{X})=\HH^{3}(X,\cO_{X})=0$. For instance, we have Grassmannians and flag varieties (\cite[Chapter 4]{Weyman}, complete smooth projective toric varieties (\cite[Chapter 9]{CLS}) or 
arithmetically Cohen-Macaulay varieties.

ii) On the other hand, when $\dim(X)=3$ the condition $\HH^{3}(X,\cO_{X}(-L))=0$ is not always satisfied even in the case of complete smooth toric varieties, as the following example shows: take $X=\PP^{3}$ and $L=\cO_{\PP^{3}}(i)$ with $i\geq4$.
When $X$ is a complete toric variety of dimension $n$, this technical condition can be tackled using the Batirev-Borisov vanishing theorem. Indeed, for any Cartier nef divisor $D$ on $X$ there is a lattice polytope $P_{D}$ such that
\[
\HH^{0}(X,\cO_{X}(D))=|P_{D}|
\]
and
\[
\hh^{i}(X,\cO_{X}(-D)=
\left\{
\begin{array}{ll}
    0, &  i\neq\dim(P_{D})\\
    |\relint( P_{D})|, & i=\dim(P_{D}),
\end{array}
\right.
\]
where $|P_{D}|$ is the number of lattice points of $P_{D}$ and $|\relint(P_{D})|$ is the number of lattice points in the relative interior of the polytope $P_{D}$
(see \cite[Section 5 and Theorem 9.2.7]{CLS}).
\end{remark}

Theorem 3.1 has been recently generalized in \cite{FM} and \cite{FM1} where using generalized syzygy bundles we 
construct recursively open subspaces of moduli spaces of simple
sheaves on X that are smooth, rational, quasiprojective varieties.

 At the end of Section \ref{Section:stability of non-complete} we will apply the results of this section (see Theorem \ref{Theorem:moduli_PnxPm}) to general syzygy bundles on $\PP^{m}\times\PP^{n}$ associated to linear systems.

\section{Stability of syzygy bundles of non-complete linear systems}\label{Section:stability of non-complete}
The aim of this Section is to answer Question \ref{Question:Brenner} for $X=\PP^{m}\times\PP^{n}$. Notice that $X$ may be viewed as the smooth complete toric variety toric variety with Cox ring 
\[
R=\kk[x_{0},\dotsc,x_{m},y_{0},\dotsc,y_{n}]
\]
graded by 
\begin{align*}
    &\deg(x_{i})=(1,0)&&0\leq i\leq m\\
    &\deg(y_{i})=(0,1)&&0\leq i\leq n.
\end{align*}
Any line bundle on $X$ is of the form $\cO_{X}(a,b):=\pi_{1}^{\ast}\cO_{\PP^{m}}(a)\otimes\pi_{2}^{\ast}\cO_{\PP^{n}}(b)$ for some integers $a,b$. A line bundle $\cO_{X}(a,b)$ is ample if and only if it is very ample, if and only if $a,b>0$; and it is effective if and only if $a,b\geq0$. Moreover, we have the following identification of vector spaces:
\[
\HH^{0}(X,\cO_{X}(a,b))\cong R_{(a,b)}:=\langle x_{0}^{\alpha_{0}}\dotsb x_{m}^{\alpha_{m}}y_{0}^{\beta_{0}}\dotsb y_{n}^{\beta_{n}}\mid\alpha_{0}+\dotsb+\alpha_{m}=a,\;\beta_{0}+\dotsb+\beta_{n}=b\rangle.
\]
In particular, for all integers $a,b\geq0$ we have
\[
\dim \HH^{0}(X,\cO_{X}(a,b))=\binom{m+a}{m}\binom{n+b}{n}.
\]

In this setting, Lemma \ref{key} can be rephrased as follows:
\begin{lemma}\label{Lemma:key rephrased}
Take $X=\PP^{m}\times\PP^{n}$ and $L=\cO_{X}(a,b)$ a very ample line bundle. Let $V\subset \HH^{0}(X,L)$ be a vector space such that $\cL_{V}$ is a base point-free linear system. The syzygy bundle $M_{V}$ is $L-$stable if for any $0<q<r-1$ and any line bundle $G=\cO_{X}(x,y)$, then 
\[
\HH^{0}(X,\bigwedge^{q}M_{V}(x,y))\neq0 \quad \text{implies} \quad 
bmx+any>\frac{qab(m+n)}{r-1}.
\]
\end{lemma}
\begin{proof}
It follows from Lemma \ref{key} using that 
\[
L^{m+n}=a^{m}b^{n}\binom{m+n}{m}=a^{m}b^{n}\binom{m+n-1}{n}(\frac{m}{n}+1)
\]
and
\[
G\cdot L=a^{m-1}b^{n-1}\binom{m+n-1}{m}(b\frac{m}{n}x+ay).
\]
\end{proof}

Notice that if $G=\cO(x,y)$ is a line bundle we have, from (\ref{Eq:Syzygy bundle exact sequence}), the following inclusion of vector spaces
\[
\HH^{0}(X,\bigwedge^{q}M_{V}(x,y))\hookrightarrow \HH^{0}(X,\bigwedge^{q}(V\otimes\cO_{X})(x,y))=\HH^{0}(X,(\bigwedge^{q} V)\otimes\cO_{X}(x,y)).
\]
Hence, if $\HH^{0}(X,\bigwedge^{q}M_{V}(x,y))\neq0$, then $G=\cO_{X}(x,y)$ is effective. Therefore, we have $x,y\geq0$. 

Our first goal is to prove that the syzygy bundle $M_{V}$ associated to any sufficiently large vector space $V\subset \HH^{0}(X,L)$ is always $L-$stable. More precisely, we show that for any base point-free linear system $\cL_{V}$ associated to an $r-$dimensional vector space $V\subset \HH^{0}(X,L)$ such that 
\begin{equation}\label{Eq:Range A}
\frac{a(m+n)}{\min(m,n)}<r\leq \binom{a+m}{m}\binom{b+n}{n},
\end{equation}
the syzygy bundle $M_{V}$ is $L-$stable. To this end, the following Lemma is needed.

\begin{lemma}
Take $X=\PP^{m}\times\PP^{n}$ and $L=\cO_{X}(a,b)$ a very ample line bundle. Let $V\subset \HH^{0}(X,L)$ be a vector space such that $\cL_{V}$ is a base point-free linear system. For any $0<q<r-1$ and any line bundle $G=\cO_{X}(x,y)$, if $\HH^{0}(\bigwedge^{q}M_{V}(x,y))\neq 0$, then $x+y\geq q$.

\end{lemma}
\begin{proof} We have that 
\[
M_{V}(-L)=\widetilde{K_{V}},
\]
where $K_{V}=\syz(f_{1},\dotsc,f_{r})$ is the syzygy module of the forms $\{f_{1},\dotsc,f_{r}\}\subset R_{(a,b)}$ corresponding to a basis of $V$. Then, $K_{V}$ lies in the following exact sequence:
\[
0\longrightarrow K_{V}\longrightarrow R(-a,-b)^{r}\xrightarrow{(f_{1},\dotsc,f_{r})}R\longrightarrow0.
\]
In particular, if
\[
\bigoplus_{i=1}^{\mu}R(-a-\alpha_{i},-b-\beta_{i})
\longrightarrow K_{V}
\longrightarrow0
\]
is a minimal presentation of $K_{V}$, then we have for any $1\leq i\leq \mu$ that
\[
(\alpha_{i},\beta_{i})\in\{(x,y)\in\NN^{2}\mid x+y\geq1\}.
\]
As a consequence, we have a minimal presentation
\begin{equation}\label{Eq:Minimal presentation of syzbundle}
\bigoplus_{i=1}^{\mu}\cO(-\alpha_{i},-\beta_{i})
\longrightarrow M_{V}
\longrightarrow0
\end{equation}
such that $\alpha_{i}+\beta_{i}\geq1$ for $1\leq i\leq \mu$. Taking exterior powers in (\ref{Eq:Minimal presentation of syzbundle}) we have
\begin{equation}\label{Eq:Minimal presentation wedge syzbundle}
    \bigoplus_{
    1\leq i_{1}<\dotsb<i_{q}\leq \mu
    }
    \cO(
    -\alpha_{i_{1}}-\dotsb-\alpha_{i_{q}},-\beta_{i_{1}}-\dotsb-\beta_{i_{q}}
    )
    \longrightarrow 
    \bigwedge^{q}M_{V}
    \longrightarrow 0.
\end{equation}
Notice that for any $q-$uple $1\leq i_{1}<\dotsc<i_{q}\leq \mu$ it holds that
\[
\alpha_{i_{1}}+\dotsb+\alpha_{i_{q}}+\beta_{i_{1}}+\dotsb+\beta_{i_{q}}\geq q.
\]
Thus, if $\HH^{0}(X,\bigwedge^{q}M_{V}(x,y))\neq0$, then we have that $x+y\geq q$.
\end{proof}
\begin{proposition}\label{Proposition:range 1}
Take $X=\PP^{m}\times\PP^{n}$ and $L=\cO_{X}(a,b)$ a very ample line bundle such that $a\geq b$. Let $\cL_{V}$ be any basepoint-free linear system associated to a vector subspace $V\subset \HH^{0}(X,L)$ with $\dim(V)=r$. If 
\[
\frac{a(m+n)}{\min(m,n)}+1<r\leq \binom{a+m}{m}\binom{b+n}{n},
\]
then the syzygy bundle $M_{V}$ associated to $\cL_{V}$ is $L-$stable.
\end{proposition}
\begin{proof}
We apply Lemma \ref{Lemma:key rephrased}. Let us consider an integer $0<q<r-1$ and $G=\cO_{X}(x,y)$ a line bundle such that $\HH^{0}(X,\bigwedge^{q}M_{V}(x,y))\neq0$. Since $x,y\geq0$ and $a\geq b$, we have that $bmx+any\geq b\min(m,n)(x+y)$. Therefore, it is enough to see that it holds
\[
bnx+amy>\frac{qab(m+n)}{r-1}.
\]

Since $\HH^{0}(X,\bigwedge^{q}M_{V}(x,y))\neq0$ then we have that there is some $q-$uple $(i_{1},\dotsc,i_{q})$ with $1\leq i_{1}<\dotsb<i_{q}\leq \mu$ in (\ref{Eq:Minimal presentation wedge syzbundle}), such that 
\[
(x,y)\in(\alpha_{i_{1}}+\dotsb+\alpha_{i_{q}},\beta_{i_{1}}+\dotsb+\beta_{i_{q}})+\ZZ^{2}_{\geq0}.
\]
Thus, we have that $x+y\geq q$. Consequently,
\[
bmx+any\geq b\min(m,n)q > \frac{qab(m+n)}{r-1}
\]
since we assume that $r>\frac{a(m+n)}{\min(m,n)}+1$.
\end{proof}

The remaining of this section is devoted to show that for any integer $r$ such that
\begin{equation}\label{Eq:Range B}
\frac{a(m+n)}{b\min(m,n)} +1 < r\leq \frac{a(m+n)}{\min(m,n)},
\end{equation}
there is a vector subspace $V\subset \HH^{0}(X,L)$ with $\dim(V)=r$ such that the associated linear system $\cL_{V}$ is base point-free and the corresponding syzygy bundle $M_{V}$ is $L-$stable. 

\begin{notation}\label{Notation:tr}
For any integer $r$ satisfying (\ref{Eq:Range B}), we denote by $t_{r}$ the only integer such that
\[
	\frac{a(m+n)}{t_{r}\min(m,n)}+1<r\leq \frac{a(m+n)}{(t_{r}-1)\min(m,n)}+1.
\]
Notice that it holds $2\leq t_{r}\leq b$.
\end{notation}

The following lemma is in the core of the proof. 

\begin{lemma}\label{Lemma:cas A}
Let $R=\kk[x_{0},\dotsc,x_{m},y_{0},\dotsc,y_{n}]$ be the Cox ring of $X=\PP^{m}\times\PP^{n}$, and let $a\geq b\geq 2$ be two integers. For any integer $r$ such that
\[
\frac{a(m+n)}{b\min(m,n)} +1 < r\leq \frac{a(m+n)}{\min(m,n)},
\]
there is a vector subspace $W\subset R_{(a,b)}$ of dimension
\[
N=\dim W \geq  \frac{a(m+n)}{(t_{r}-1)\min(m,n)}+1,
\]
such that
\begin{itemize}
    \item[i)] $W$ admits a basis of monomials $\{f_{1},\dotsc,f_{N}\}$.
    \item[ii)] We have that
    \[
    \{x_{i}^{a}y_{j}^{b}\mid 0\leq i\leq m,\;0\leq j\leq n\}\subset W.
    \]
    \item[iii)] If $\xi=(g_{1},\dotsc,g_{N})\in\syz(f_{1},\dotsc,f_{N})$ is a syzygy of degree $(a+\alpha,b+\beta)$, then it holds that
    \[
    \alpha+\beta\geq t_{r}.
    \]
\end{itemize}
In particular, condition ii) implies that the linear system $\cL_{W}$ associated to $W$ is base point-free.
\end{lemma}
\begin{proof}
We divide the proof in four cases:
\begin{itemize}
\item[A)] $n=m=1$, and $R=\kk[x_{0},x_{1},y_{0},y_{1}]$.

\item[B)] $m=\min(n,m)=1$ and $n>1$, and $R=\kk[x_{0},\dotsc,x_{m},y_{0},y_{1}]$.

\item[C)] $n=\min(n,m)=1$ and $m>1$, and $R=\kk[x_{0},x_{1},y_{0},\dotsc,y_{n}]$.

\item[D)] Otherwise.
\end{itemize}

Case A). We consider the following two sets of monomials in $R_{(a,b)}$:
\begin{multline*}
    \mathcal{A}=\{
    x_{0}^{a-(t_{r}-1)}x_{1}^{t_{r}-1}y_{0}^{b-1}y_{1},
    x_{0}^{a-2(t_{r}-1)}x_{1}^{2(t_{r}-1)}y_{0}^{b},
    \dotsc,\\
    x_{0}^{a-(\lfloor\frac{a}{t_{r}-1}\rfloor-2)(t_{r}-1)}
    x_{1}^{(\lfloor\frac{a}{t_{r}-1}\rfloor-2)(t_{r}-1)}
    y_{0}^{b-(1-\epsilon)}y_{1}^{1-\epsilon},\\
    x_{0}^{a-(\lfloor\frac{a}{t_{r}-1}\rfloor -1)(t_{r}-1)}
    x_{1}^{(\lfloor\frac{a}{t_{r}-1}\rfloor -1)(t_{r}-1)}
    y_{0}^{b-\epsilon}y_{1}^{\epsilon}
    \},
\end{multline*}
and
\begin{multline*}
    \mathcal{B}=\{
    x_{0}^{a-1}x_{1}y_{0}^{b-(t_{r}-1)}y_{1}^{t_{r}-1},
    x_{0}^{a-t_{r}}x_{1}^{t_{r}}y_{0}^{b-t_{r}}y_{1}^{t_{r}},
    x_{0}^{a-2t_{r}+1}x_{1}^{2t_{r}-1}y_{0}^{b-(t_{r}-1)}y_{1}^{t_{r}-1},
    \dotsc,\\
    x_{0}^{a-1-(\lfloor\frac{a}{t_{r}-1}\rfloor-2)(t_{r}-1)}
    x_{1}^{1+(\lfloor\frac{a}{t_{r}-1}\rfloor-2)(t_{r}-1)}
    y_{0}^{b-t_{r}+\epsilon}y_{1}^{t_{r}-\epsilon},\\
    x_{0}^{a-1-(\lfloor\frac{a}{t_{r}-1}\rfloor -1)(t_{r}-1)}
    x_{1}^{1+(\lfloor\frac{a}{t_{r}-1}\rfloor -1)(t_{r}-1)}
    y_{0}^{b+t_{r}-1-\epsilon}y_{1}^{t_{r}-1+\epsilon}
    \},
\end{multline*}
where
\[
\epsilon=\epsilon(a,t_{r}):=\left\{
\begin{array}{cc}
    0 & \text{if}\quad\lfloor \frac{a}{t_{r}-1}\rfloor\quad \text{even} \\
    1 & \text{if}\quad\lfloor \frac{a}{t_{r}-1}\rfloor\quad \text{odd}.
\end{array}
\right.
\]
We have that
\begin{align*}
    &|\cA|=\lfloor\frac{a}{t_{r}-1}\rfloor-1\\
    &|\cB|=\lfloor\frac{a}{t_{r}-1}\rfloor
\end{align*}
We construct the vector space $W$ from the sets of monomials $\cA$ and $\cB$, fulfilling conditions i), ii) and iii). We consider two subcases: 

Case i): if $2(t_{r}-1)\leq b$.

If either it holds $\frac{a}{t_{r}-1}\notin\ZZ$ or $\frac{a}{t_{r}-1}\in\ZZ$ is odd, we define
\[
W=\langle\cA\cup\cB\rangle\oplus\langle x_{i}^{a}y_{j}^{b}\mid 0\leq i,j\leq 1\rangle.
\]
We have $\dim W=2\lfloor\frac{a}{t_{r}-1}\rfloor + 3$. Otherwise, $\frac{a}{t_{r}-1}\in\ZZ$ is even, and we define
\[
W=\langle(\cA\setminus\{x_{0}^{t_{r}-1}
    x_{1}^{a-(t_{r}-1)}
    y_{0}^{b}\})\cup\cB\rangle\oplus\langle x_{i}^{a}y_{j}^{b}\mid 0\leq i,j\leq 1\rangle.
\]
We have $\dim W=2\lfloor\frac{a}{t_{r}-1}\rfloor + 2$.

Case ii): if $t_{r}\leq b< 2(t_{r}-1)$. 

If either it holds $\frac{a}{t_{r}-1}\notin\ZZ$ or $\frac{a}{t_{r}-1}\in\ZZ$ is odd, we define
\[
W=\langle\cA\cup(\cB\setminus \{ x_{0}^{a-1}x_{1}y_{0}^{b-(t_{r}-1)}y_{1}^{t_{r}-1}\})\rangle\oplus
\langle x_{i}^{a}y_{j}^{b}\mid 0\leq i,j\leq 1\rangle.
\]
We have $\dim W=2\lfloor\frac{a}{t_{r}-1}\rfloor + 2$. Otherwise, $\frac{a}{t_{r}-1}\in\ZZ$ is even, and we define
\[
W=\langle(\cA\setminus\{x_{0}^{t_{r}-1}
    x_{1}^{a-(t_{r}-1)}
    y_{0}^{b}\})\cup(\cB\setminus \{ x_{0}^{a-1}x_{1}y_{0}^{b-(t_{r}-1)}y_{1}^{t_{r}-1}\})\rangle\oplus\langle x_{i}^{a}y_{j}^{b}\mid 0\leq i,j\leq 1\rangle.
\]
We have $\dim W=2\lfloor\frac{a}{t_{r}-1}\rfloor + 1$.

In any case, we have that $\dim W\geq \frac{2a}{t_{r}-1}+1$ and conditions i), ii) and iii) hold.

Case B). We consider the following sets of monomials in $R_{(a,b)}$ for each $1\leq i\leq m-1$:
\begin{multline*}
    \mathcal{A}_{i}=\{
    x_{i-1}^{a-(t_{r}-1)}
    x_{i}^{t_{r}-1}
    y_{0}^{b-1}
    y_{1},
    x_{i-1}^{a-2(t_{r}-1)}
    x_{i}^{2(t_{r}-1)}
    y_{0}^{b},
    \dotsc,\\
    x_{i-1}^{a-(\lfloor\frac{a}{t_{r}-1}\rfloor-2)(t_{r}-1)}
    x_{i}^{(\lfloor\frac{a}{t_{r}-1}\rfloor-2)(t_{r}-1)}
    y_{0}^{b-(1-\epsilon)}
    y_{1}^{1-\epsilon},\\
    x_{i-1}^{a-(\lfloor\frac{a}{t_{r}-1}\rfloor -1)(t_{r}-1)}
    x_{i}^{(\lfloor\frac{a}{t_{r}-1}\rfloor -1)(t_{r}-1)}
    y_{0}^{b-\epsilon}
    y_{1}^{\epsilon}
    \},
\end{multline*}
and
\begin{multline*}
    \mathcal{B}_{i}=\{
    x_{i}^{a-(t_{r}-1)}
    x_{i+1}^{t_{r}-1}
    y_{0}^{b-2}
    y_{1}^{2},
    x_{i}^{a-2(t_{r}-1)}
    x_{i+1}^{2(t_{r}-1)}
    y_{0}^{b-1}y_{1},
    \dotsc,\\
    x_{i}^{a-(\lfloor\frac{a}{t_{r}-1}\rfloor-2)(t_{r}-1)}
    x_{i+1}^{(\lfloor\frac{a}{t_{r}-1}\rfloor-2)(t_{r}-1)}
    y_{0}^{b-(2-\epsilon)}
    y_{1}^{2-\epsilon},\\
    x_{i}^{a-(\lfloor\frac{a}{t_{r}-1}\rfloor -1)(t_{r}-1)}
    x_{i+1}^{(\lfloor\frac{a}{t_{r}-1}\rfloor -1)(t_{r}-1)}
    y_{0}^{b-(1+\epsilon)}
    y_{1}^{1+\epsilon}
    \},
\end{multline*}

If $\frac{a}{t_{r}-1}\notin\ZZ$ or $\frac{a}{t_{r}-1}\in\ZZ$ and it is odd, we define
\[
W=\bigoplus_{i=1}^{m}\langle\cA_{i}\cup\cB_{i}\rangle\oplus
\langle x_{i}^{a}y_{j}^{b}\mid 0\leq i\leq m,\;0\leq j\leq 1\rangle,
\]
and we have
\[
\dim W = 2(m-1)(\lfloor\frac{a}{t_{r}-1}\rfloor-1)+2(m+1).
\]

Otherwise, $\frac{a}{t_{r}-1}\in\ZZ$ and it is even, and we define
\[
W=\bigoplus_{i=1}^{m}\langle(\cA_{i}
\setminus 
\{
x_{i-1}^{t_{r}-1}
    x_{i}^{a-(t_{r}-1)}
    y_{0}^{b}
\}
)\cup\cB_{i}\rangle\oplus
\langle x_{i}^{a}y_{j}^{b}\mid 0\leq i\leq m,\;0\leq j\leq 1\rangle,
\]
and we have
\[
\dim W = (m-1)(\lfloor\frac{a}{t_{r}-1}\rfloor-2+\lfloor\frac{a}{t_{r}-1}\rfloor-1)+2(m+1).
\]
In any case, we have $\dim W \geq \frac{(m+1)a}{t_{r}-1}+1$ and conditions i), ii) and iii) hold.

Case C). We consider the following sets of monomials in $R_{(a,b)}$ for each $1\leq i\leq n-1$:
\begin{multline*}
    \mathcal{A}_{i}=\{
    x_{0}^{a-(t_{r}-1)}
    x_{1}^{t_{r}-1}
    y_{i}^{b-1}
    y_{i+1},
    x_{0}^{a-2(t_{r}-1)}
    x_{1}^{2(t_{r}-1)}
    y_{i}^{b},
    \dotsc,\\
    x_{0}^{a-(\lfloor\frac{a}{t_{r}-1}\rfloor-2)(t_{r}-1)}
    x_{1}^{(\lfloor\frac{a}{t_{r}-1}\rfloor-2)(t_{r}-1)}
    y_{i}^{b-(1-\epsilon)}
    y_{i+1}^{1-\epsilon},\\
    x_{0}^{a-(\lfloor\frac{a}{t_{r}-1}\rfloor -1)(t_{r}-1)}
    x_{1}^{(\lfloor\frac{a}{t_{r}-1}\rfloor -1)(t_{r}-1)}
    y_{i}^{b-\epsilon}
    y_{i+1}^{\epsilon}
    \},
\end{multline*}
and for $i=n$:
\begin{multline*}
    \mathcal{A}_{n}=\{
    x_{0}^{a-(t_{r}-1)}
    x_{1}^{t_{r}-1}
    y_{n}^{b-1}
    y_{0},
    x_{0}^{a-2(t_{r}-1)}
    x_{1}^{2(t_{r}-1)}
    y_{n}^{b},
    \dotsc,\\
    x_{0}^{a-(\lfloor\frac{a}{t_{r}-1}\rfloor-2)(t_{r}-1)}
    x_{1}^{(\lfloor\frac{a}{t_{r}-1}\rfloor-2)(t_{r}-1)}
    y_{n}^{b-(1-\epsilon)}
    y_{0}^{1-\epsilon},\\
    x_{0}^{a-(\lfloor\frac{a}{t_{r}-1}\rfloor -1)(t_{r}-1)}
    x_{1}^{(\lfloor\frac{a}{t_{r}-1}\rfloor -1)(t_{r}-1)}
    y_{n}^{b-\epsilon}
    y_{0}^{\epsilon}
    \},
\end{multline*}

If $\frac{a}{t_{r}-1}\notin\ZZ$ or $\frac{a}{t_{r}-1}\in\ZZ$ and it is odd, we define
\[
W=\bigoplus_{i=1}^{n}\langle\cA_{i}\rangle\oplus
\langle x_{i}^{a}y_{j}^{b}\mid 0\leq i\leq 1,\;0\leq j\leq n\rangle,
\]
and we have
\[
\dim W = (n+1)(\lfloor\frac{a}{t_{r}-1}\rfloor-1)+2(n+1).
\]

Otherwise, $\frac{a}{t_{r}-1}\in\ZZ$ and it is even, and we define
\[
W=\bigoplus_{i=1}^{n}\langle\cA_{i}
\setminus 
\{
x_{0}^{t_{r}-1}
    x_{1}^{a-(t_{r}-1)}
    y_{i}^{b}
\}
\rangle\oplus
\langle x_{i}^{a}y_{j}^{b}\mid 0\leq i\leq n,\;0\leq j\leq 1\rangle,
\]
and we have
\[
\dim W = (n+1)(\lfloor\frac{a}{t_{r}-1}\rfloor-2)+2(n+1).
\]
In any case, we have $\dim W \geq \frac{(n+1)a}{t_{r}-1}+1$ and conditions i), ii) and iii) hold.

Case D).
Let us consider now the vector space $W$ constructed in Case B) (respectively Case C)) taking the subring $R'=\kk[x_{0},\dotsc,x_{m},y_{0},y_{1}]\subset R$, if $m=\max(m,n)$ (respectively taking the subring $R'=\kk[x_{0},x_{1},y_{0},\dotsc,y_{n}]$, if $n=\max(m,n)$). Thus, $W$ is also a vector subspace in $R_{(a,b)}$ and conditions $i)$, $ii)$ and $iii)$ are automatically satisfied. On the other hand, we have that
\[
\begin{array}{rc>{\displaystyle}l}
\dim W
&\geq &
\frac{a(\max(m,n)+1)}{t_{r}-1}+1\\[2mm]
&=&
\frac{\max(m,n)a}{t_{r}-1}+\frac{a}{t_{r}-1}+1\\[2mm]
&\geq&
\frac{\max(m,n)a}{\min(m,n)(t_{r}-1)}+\frac{a}{t_{r}-1}+1\\[2mm]
&=&
\frac{(m+n)a}{\min(m,n)(t_{r}-1)}+1.
\end{array}
\]
\end{proof}

\begin{proposition}\label{Proposition:range 2}
Take $X=\PP^{m}\times\PP^{n}$ and $L=\cO_{X}(a,b)$ a very ample line bundle, such that $a\geq b\geq 2$. For any integer
\[
\frac{a(m+n)}{b\min(m,n)}+1<r\leq \frac{a(m+n)}{\min(m,n)}+1,
\]
let $\cL_{V}$ be general base point-free linear system associated to an $r-$dimensional vector space $V\subset\HH^{0}(X,L)$. Then, the syzygy bundle $M_{V}$ corresponding to $\cL_{V}$ is $L-$stable.
\end{proposition}
\begin{proof}
By Remark \ref{Remark:stability open property} it is enough to find an $r-$dimensional vector space $V\subset \HH^{0}(X,L)$ such that the linear system $\cL_{V}$ is base point-free and its corresponding syzygy bundle $M_{V}$ is $L-$stable. To this end we use Lemma \ref{Lemma:cas A}.

By Notation \ref{Notation:tr}, let us consider the integer $2\leq t_{r}\leq b$ such that
\[
\frac{a(m+n)}{t_{r}\min(m,n)}+1<r\leq \frac{a(m+n)}{(t_{r}-1)\min(m,n)}+1.
\]
We consider the $N-$dimensional vector subspace $W$ given by Lemma \ref{Lemma:cas A}. We may write
\[
W=\langle f_{1},\dotsc,f_{N}\rangle=\langle 
x_{i}^{a}y_{j}^{b}\mid 0\leq i\leq m,\;0\leq j\leq n
\rangle+
\langle
m_{1},\dotsc,m_{N-(m+1)(n+1)}
\rangle,
\]
where $m_{i}$ is a monomial for any $1\leq i\leq N-(m+1)(n+1)$. Then we construct the following $r-$dimensional vector space
\[
V=\langle f_{1},\dotsc,f_{r}\rangle=\langle 
x_{i}^{a}y_{j}^{b}\mid 0\leq i\leq m,\;0\leq j\leq n
\rangle+
\langle
m_{1},\dotsc,m_{r-(m+1)(n+1)}
\rangle.
\]
Notice that $\syz(f_{1},\dotsc,f_{r})\subset \syz(f_{1},\dotsc,f_{N})$. Thus, applying condition $ii)$ of Lemma \ref{Lemma:cas A} to any syzygy $\xi=(g_{1},\dotsc,g_{r})\in\syz(V)$ of degree $(a+\alpha,b+\beta)$, we have that
\[
\alpha+\beta\geq t_{r}.
\]
In particular, if $K_{V}:=\syz(f_{1},\dotsc,f_{r})$, we have the following minimal presentation
\[
\bigoplus_{i=1}^{\mu}R(-a-\alpha_{i},-b-\beta_{i})\rightarrow K_{V}
\longrightarrow0,
\]
such that $\alpha_{i}+\beta_{i}\geq t_{r}$ for any $1\leq i\leq \mu$.

On the other hand, since $M_{V}=\widetilde{K_{V}}$, we have the following minimal presentation
\[
\bigoplus_{i=1}^{\mu}\cO_{X}(-\alpha_{i},-\beta_{i})\longrightarrow M_{V}\longrightarrow 0,
\]
which yields, taking the $q-$th exterior power, the minimal presentation of $\bigwedge^{q}M_{V}$:
\[
\bigoplus_{1\leq i_{1}<\dotsb<i_{q}\leq \mu}\cO_{X}(
-\alpha_{i_{1}}-\dotsb-\alpha_{i_{q}},-\beta_{i_{1}}-\dotsb-\beta_{i_{q}})
\longrightarrow \bigwedge^{q}M_{V}\longrightarrow 0.
\]
For any $q-$uple $1\leq i_{1}<\dotsb<i_{q}\leq \mu$, we have that
\begin{equation}\label{Eq:bound cas A}
-\alpha_{i_{1}}-\dotsb-\alpha_{i_{q}}-\beta_{i_{1}}-\dotsb-\beta_{i_{q}}\geq qt_{r}.
\end{equation}
Now, we apply Lemma \ref{key}. Let us consider an integer $0<q<r-1$ and a line bundle $G=\cO_{X}(x,y)$ such that 
\[
\HH^{0}(X,\bigwedge^{q}M_{V}(x,y))\neq 0.
\]
We want to see that it holds
\[
bmx+any>\frac{qab(m+n)}{r-1}.
\]
Since $a\geq b$, we have that $bmx+any\geq b\min(m,n)(x+y)$. On the other hand, by (\ref{Eq:bound cas A}), assuming that $\HH^{0}(X,\bigwedge^{q}M_{V}(x,y))\neq0$ we have that $x+y\geq qt$. Consequently, we obtain that
\[
bmx+any\geq b\min(m,n)(x+y)\geq bqt_{r} >\frac{qab(m+n)}{r-1}
\]
since we assume that 
\[
r>\frac{a(m+n)}{t_{r}\min(m,n)}+1.
\]
\end{proof}

Propositions \ref{Proposition:range 1} and \ref{Proposition:range 2} yield the main result of this note.
\begin{theorem}\label{Theorem:main}
Let $X=\PP^{m}\times\PP^{n}$ and $a,b\geq1$ two integers and let $L:=\cO_{X}(a,b)$ be a very ample line bundle on $X$. For any integer $r$ such that
\[
\frac{\max(a,b)(m+n)}{\min(a,b)\min(m,n)}+1<r\leq \binom{m+a}{m}\binom{n+b}{n}=\hh^{0}(X,L),
\]
let $\cL_{V}$ be a general base point-free linear system associated to an $r-$dimensional vector space $V\subset \HH^{0}(X,L)$. Then, the syzygy bundle $M_V$ corresponding to $\cL_{V}$ is $L-$stable.
\end{theorem}
\begin{proof}
Exchanging the role of $m$ and $n$ in $\PP^{m}\times\PP^{n}$ if necessary, we may assume that $a\geq b$. On the other hand, if we assume $b=1$, then we have
\[
\frac{a(m+n)}{b\min(m,n)}+1=\frac{a(m+n)}{\min(m,n)}+1.
\]
Therefore, the result follows directly from Proposition \ref{Proposition:range 1}. 

On the other hand, assume that $a\geq b\geq2$.
Then, the result follows from Proposition \ref{Proposition:range 1} if
\[
\frac{a(m+n)}{\min(m,n)}+1<r\leq \binom{m+a}{m}\binom{n+b}{n},
\]
or from Proposition \ref{Proposition:range 2} if
\[
\frac{a(m+n)}{b\min(m,n)}+1<r\leq \frac{a(m+n)}{\min(m,n)}+1.
\]
\end{proof}

We have the following remark:

\begin{remark}\label{Remark:Gaps}\rm
(i) Let $V\subset \HH^{0}(X,\cO_{X}(L))$ be a vector space corresponding to a base point-free linear system $\cL_{V}$, then we have that
\[
\dim X=m+n < \dim V \leq \binom{m+a}{m}\binom{n+b}{n}=\hh^{0}(X,L).
\]
By Theorem \ref{Theorem:main}, we have seen that when $V$ is general and
\[
\frac{\max(a,b)(m+n)}{\min(a,b)\min(m,n)}+1<\dim V\leq \binom{m+a}{m}\binom{n+b}{n}=\hh^{0}(X,L),
\]
then, the corresponding syzygy bundle $M_{V}$ is $L-$stable. However, to solve completely Question \ref{Question:Brenner}, it remains open the case of a general vector space $V$ such that
\begin{equation}\label{Eq:Gap}
m+n+1\leq \dim V\leq \frac{\max(a,b)(m+n)}{\min(a,b)\min(m,n)}+1.
\end{equation}

(ii) On the other hand, if we assume in addition that $V$ is generated by monomials, then being $\cL_{V}$ a base point-free linear system implies that
\[
(m+1)(n+1)\leq \dim V\leq \binom{m+a}{m}\binom{n+b}{n}.
\]
Therefore, Question \ref{Question:Brenner} even when restricted to vector spaces $V$ generated by monomials remains open when
\begin{equation}\label{Eq:Gap monomials}
(m+1)(n+1)\leq \dim V \leq \frac{\max(a,b)(m+n)}{\min(a,b)\min(m,n)}+1.
\end{equation}
However, we notice that the open range of cases expressed in (\ref{Eq:Gap monomials}) is smaller than that of (\ref{Eq:Gap}).
\end{remark}

The following corollaries shows that in some cases Theorem \ref{Theorem:main} solves already Question \ref{Question:Brenner} for $\PP^{m}\times \PP^{n}$.

\begin{corollary}
Let $X=\PP^{m}\times\PP^{n}$ and $a,b\geq1$ two integers and let $L:=\cO_{X}(a,b)$ be a very ample line bundle on $X$. If
\[
\max(a,b) < \min(m,n)\min(a,b),
\]
then the syzygy bundle $M_{V}$ corresponding to a general basepoint-free linear system $\cL_{V}$, is $L-$stable.

In particular, if $m,n\geq2$, for any integer $t\geq1$ we set $H_{t}:=\cO_{X}(t,t)$. Then, the syzygy bundle $M_{V}$ associated to a general basepoint-free linear system $\cL_{V}$ with $V\subset\HH^{0}(X,H_{t})$ is $H_{t}-$stable. Notice that then, $M_{V}$ is also $H_{1}-$stable.
\end{corollary}
\begin{proof}
In this case, we have
\[
m+n+1>\frac{\max(a,b)(m+n)}{\min(a,b)\min(m,n)}+1.
\]
Since there is no base point-free linear system associated to a vector space $V$ satisfying (\ref{Eq:Gap}), the result follows from Theorem \ref{Theorem:main}.
\end{proof}

\begin{corollary}\label{Corollary:PnxPm}
Let $X=\PP^{m}\times\PP^{n}$ and $a,b\geq1$ two integers and let $L:=\cO_{X}(a,b)$ be a very ample line bundle on $X$. If
\[
\max(a,b) < \frac{(mn+m+n)\min(m,n)}{m+n}\min(a,b),
\]
then the syzygy bundle $M_{V}$ corresponding to a general base point-free linear system $\cL_{V}$ given by a general vector space $V$ generated by monomials, is $L-$stable.
\end{corollary}
\begin{proof}
In this case, we have
\[
(m+1)(n+1)>\frac{\max(a,b)(m+n)}{\min(a,b)\min(m,n)}+1.
\]
Since there is no base point-free linear system associated to a vector space $V$ generated by monomials satisfying (\ref{Eq:Gap monomials}), the result follows from Theorem \ref{Theorem:main}.
\end{proof}

\begin{remark}\rm
Let us consider a {\em multiprojective space} $X=\PP^{m_{1}}\times\dotsb\times\PP^{m_{k}}$, polarized by a very ample line bundle $L:=\cO_{X}(a_{1},\dotsc,a_{k})$, $a_i>0$. We notice that one can use analogous techniques of the previous results to obtain an integer $B(m_{1},\dotsc,m_{k};a_{1},\dotsc,a_{k})\geq m_{1}+\dotsb+m_{k}+1$ such that for a general subspace $V\subset \HH^{0}(X,\cO_{X}(L))$ satisfying
\[
B(m_{1},\dotsc,m_{k};a_{1},\dotsc,a_{k})\leq \dim V\leq \binom{m_{1}+a_{1}}{m_{1}}\dotsb\binom{m_{k}+a_{k}}{m_{k}},
\]
the syzygy bundle $M_{V}$ is $L-$stable.
\end{remark}

As an application of Theorem \ref{Theorem:moduli} to this setting, we address now the geometry of the moduli space $\cM$ on which an $L-$stable syzygy bundle $M_{V}$ can be represented (see Section \ref{Section:moduli}).

\begin{theorem}\label{Theorem:moduli_PnxPm}
Let $X=\PP^{m}\times\PP^{n}$, and $a,b\geq1$ two integers and let $L:=\cO_{X}(a,b)$ be a very ample line bundle on $X$. For any integer $r$ such that
\[
\frac{\max(a,b)(m+n)}{\min(a,b)\min(m,n)}+1<r\leq \binom{m+a}{m}\binom{n+b}{n}=\hh^{0}(X,L),
\]
and a general subspace $V\subset \HH^{0}(X,L)$ with $\dim V=r$, we consider $M_{V}$ a general syzygy bundle corresponding to the non-complete linear system $\cL_{V}$. Let $[M_{V}]$ be the point representing $M_{V}$ inside the moduli space $\cM$. We have
\begin{itemize}
    \item[i)] If either $m+n\geq4$, or $m+n=3$ and $r=\hh^{0}(X,\cO_{X}(L))$, then $[M_{V}]$ is a smooth point and
    \[
    \dim_{\kk}T_{[M_{V}]}\cM=r(\hh^{0}(X,\cO_{X}(L))-r).
    \]
    \item[ii)] If $m+n=2$ then $[M_{V}]$ is a smooth point and
    \[
     \dim_{\kk}T_{[M_{V}]}\cM=r(\hh^{0}(X,\cO_{X}(L))-r)+r\hh^{2}(X,\cO_{X}(-L)).
    \]
\end{itemize}
In particular $M_{V}$ is infinitesimally rigid if and only if $\cL_{V}$ is a complete linear system.
\end{theorem}
\begin{proof}
It follows straightforward from Theorem \ref{Theorem:main} and Corollary \ref{Theorem:moduli}.
\end{proof}

We finish this note with some examples and open problems. The following example shows that the bounds established in Theorem \ref{Theorem:main} are not far of being optimal. Indeed, the following example shows that for small positive integers $a,b$ we are not always able to produce a vector space $V\subset \HH^{0}(X,\cO_{X}(a,b))$ generated by monomials such that
\begin{itemize}
\item[i)] The linear system $\cL_{V}$ is base point-free.
\item[ii)] The dimension of $V$ satisfies (\ref{Eq:Gap monomials}) i.e.
\[
	(m+1)(n+1)\leq \dim V \leq \frac{\max(a,b)(m+n)}{\min(a,b)\min(m,n)}+1.
\]
\item[iii)] The corresponding syzygy bundle $M_{V}$ is $L$-stable.
\end{itemize}

\begin{example}\label{Example:non always monomial}\rm
Let $m=1$, $n=1$ and $X=\PP^{1}\times\PP^{1}$. We consider $a=2$ and $b=1$, and the very ample line bundle $L=\cO_{X}(2,1)$ on $X$. We have the identification of vector spaces
\[
\HH^{0}(X,L)\cong
\kk\langle 
x_{0}^{2}y_{0}, x_{0}^{2}y_{1}, x_{1}^{2}y_{0},x_{1}^{2}y_{1},
x_0x_1y_0, x_0x_1y_1
\rangle.
\]
The open range (\ref{Eq:Gap monomials}) for the dimension of vector spaces $V$ generated by monomials associated to base point-free linear systems $\cL_{V}$ is
\[
4=(m+1)(n+1)\leq \dim(V)\leq \frac{\max(a,b)(m+n)}{\min(a,b)\min(m,n)}+1=5
\]
We will see that there is no such subspace $V$ with $\dim V=5$ such that its corresponding syzygy bundle $M_{V}$ is $L-$stable. Indeed, we have only two possibilities:
\begin{enumerate}
\item[a)] $V_{1}=\langle x_{0}^{2}y_{0}, x_{0}^{2}y_{1}, x_{1}^{2}y_{0},x_{1}^{2}y_{1},
x_0x_1y_0\rangle$.
\item[b)] $V_{2}=x_{0}^{2}y_{0}, x_{0}^{2}y_{1}, x_{1}^{2}y_{0},x_{1}^{2}y_{1}, x_0x_1y_1\rangle$.
\end{enumerate}
Notice that they are symmetric up to permutation of the variables $\{y_{0},y_{1}\}$.

In any case one can check that $\HH^{0}(X,M_{V_{i}}(1,0))\neq 0$, then we have an inclusion of vector bundles
\[
\cO_{X}(-1,0)\hookrightarrow M_{V_{i}}.
\]
However comparing the slopes we have
\[
\mu_{L}(\cO_{X}(-1,0))=-1=-\frac{L^{2}}{\dim V_{i}-1}=\mu_{L}(M_{V_{i}}).
\]
In particular $\cO_{X}(-1,0)$ is a subbundle destabilizing both syzygy bundles $M_{V_{1}}$ and $M_{V_{2}}$.

Notwithstanding, one can check that the $4-$dimensional subspace $W_{1}=\langle 
x_{0}^{2}y_{0}, x_{0}^{2}y_{1}, x_{1}^{2}y_{0},x_{1}^{2}y_{1}\rangle$, as well as the $6-$dimensional subspace $W_{2}=\HH^{0}(X,L)$ provide base point-free linear systems whose syzygy bundles $M_{W_{i}}$ are $L-$stable. 
In summary, Example \ref{Example:non always monomial} shows that dealing only with monomials, the lower bound established in Theorem \ref{Theorem:main} cannot always be improved. Indeed, we have first seen that the case $\dim V=5$ cannot be covered.
\end{example}

As we pointed out in the introduction, it is a longstanding problem to determine the stability of the syzygy bundle associated to a complete linear system. More precisely, in \cite[Conjecture 2.6]{ELM} Ein, Lazarsfeld and Mustopa posed the following conjecture.

\begin{conjecture}\label{Conjecture:stable}
Let $A$ and $P$ two line bundles on a smooth projective variety $X$. Assume that $A$ is very ample and set $L_{d}:=dA+P$ for any positive integer $d$. Then, the syzygy bundle $M_{L_{d}}$ is $A-$stable for $d\gg0$.
\end{conjecture}

Related to this conjecture, Hering, Musta\c{t}\u{a} and Payne considered the following question.

\begin{question} \label{question} \rm
Let $L$ be an ample line bundle on a projective toric
variety $X$. Is the syzygy bundle $M_L$ semistable, with respect to some choice of polarization?
\end{question}

Based on our results we propose a generalization of the above conjecture and question, to address the stability of syzygy bundles associated to non-complete linear systems. We pose the following problem.

\begin{problem}
Let $X$ be a projective variety of dimension $d$ and $L$ be a very ample line bundle on $X$. Determine the integers $d+1\leq r\leq \hh^{0}(X,L)$ such that for a general vector space $V\subset \HH^{0}(X,L)$ of dimension $r$, the linear system $\cL_{V}$ is base point-free and the syzygy bundle $M_{V}$ is $L-$stable.
\end{problem}


\begin{thebibliography}{99}
\bibitem{Brenner} H. Brenner, {\em Looking out for stable syzygy bundles. With an appendix by Georg Hein}, Adv.
Math., {\bf 219} (2008) 401–427.

\bibitem{CL} F. Caucci and M. Lahoz, {\em Stability of syzygy bundles on abelian varieties}, Bull. Lond. Math. Soc. {\bf 53:4}  (2021) 1030-1036.

\bibitem{C} I. Coanda, {\em On the stability of syzygy bundles}, Internat. J. Math. {\bf 22} (2011),  515-534.

\bibitem{Cox} D. Cox, {\em The homogeneous coordinate ring of a toric variety}, J. Algebr. Geom. {\bf 4:1} (1995), 17--50.

\bibitem{CLS} D. Cox, J. Little and H. Schenck, {\em Toric varieties}, Graduate Studies in Mathematics {\bf 124} American Mathematical Society (2011).

\bibitem{CMM} L. Costa, P. Macias Marques and R. M. Mir\'o-Roig, {\em Stability and unobstructedness of syzygy bundles}, J. Pure Appl. Algebra {\bf 214} (2010),  1241-1262.

\bibitem{EL} L. Ein and R. Lazarsfeld, {\em  Stability and restrictions of Picard bundles, with an application to the normal bundles
of elliptic curves}. Complex projective geometry (Trieste, 1989/Bergen, 1989), 149–156, London Math. Soc. Lecture Note
Ser., {\bf 179} (1992), Cambridge Univ. Press, Cambridge.

\bibitem{ELM} L. Ein, R. Lazarsfeld and Y. Mustopa, {\em Stability of syzygy bundles on an algebraic surface}. Math. Res. Lett. {\bf 20:1} (2013), 73--80.

\bibitem{Eis-Mus-Sti} D. Eisenbud, M. Musta\c{t}\u{a} and M. Stillman {\em Cohomology on Toric Varieties and Local Cohomology with Monomial Supports}. J. Symbolic Comput. 29 (2000), no. 4-5, 583-600.

\bibitem{FM} B. ~Fantechi and R. M. ~Mir\'o-Roig, {\em Lagrangian subspaces of the moduli space of simple sheaves on K3 surfaces}, Preprint UB2023.

\bibitem{FM1} B. Fantechi and R. M. Mir\'o-Roig, {\em Moduli of syzygy bundles}, Preprint UB2023.

\bibitem{Fle} H. Flenner, {\em  Restrictions of semistable bundles on projective varieties}, Comment. Math. Helv. {\bf 59} (1984),  635-650.

\bibitem{Macaulay2} D.R. Grayson and M.E. Stillman, {\em Macaulay2, a software system for research in algebraic geometry}. Available at \url{http://www.math.uiuc.edu/Macaulay2/}.

\bibitem{G} M. Green, {\em Koszul cohomology and the geometry of projective varieties}, J. Differential Geom. {\bf 19} (1984),  125-171.

\bibitem{HMP}  M. Hering, M. Mustaţă and S. Payne, {\em Positivity for toric vector bundles},  Annales de l'Institut Fourier {\bf 60} (2010),  607-640. 


\bibitem{MM} P. Macias Marques and R. M. Mir\'o-Roig, {\em Stability of syzygy bundles}, Proc. Amer. Math. Soc.
{\bf 139} (2011),  3155-3170.


\bibitem{MS}  R. M. Mir\'o-Roig and M. Salat-Molt\'o, {\em Ein-Lazarsfeld-Mustopa conjecture for the  blow-up of a projective space}. Preprint.

\bibitem{MR} J. Mukherjee and  D. Raychaudhury, {\em A note on stability of syzygy bundles on Enriques and bielliptic surcaes}, arXiv 2111.08231.

\bibitem{T} V. Trivedi, {\em Semistability of syzygy bundles on projective spaces in positive characteristics}, Internat. J. Math. {\bf 21} (2010),  1475–1504.

\bibitem{Weyman} J. Weyman, {\em Cohomology of VectorBundles and Syzygies}. Cambridge University Press, 2003.
\end{thebibliography}
\end{document}